\documentclass[A4paper]{article}%
\usepackage{graphicx}
\usepackage{amsmath,amssymb,amsfonts}
\usepackage{theorem}
\usepackage{color}
\usepackage{hyperref}
\usepackage[font={small, it}]{caption}
\usepackage[all]{xy}%
\usepackage{amsmath}%
\usepackage{amsfonts}%
\usepackage{amssymb}
\usepackage{refcheck}
\providecommand{\U}[1]{\protect \rule{.1in}{.1in}}
\theoremstyle{change}
\newtheorem{definition}{Definition:}[section]
\newtheorem{proposition}[definition]{Proposition:}
\newtheorem{theorem}[definition]{Theorem:}
\newtheorem{lemma}[definition]{Lemma:}

{\theorembodyfont{\rmfamily}
	\newtheorem{remark}[definition]{Remark:}
}
{\theorembodyfont{\rmfamily}
	\newtheorem{example}[definition]{Example:}
}
\newenvironment{proof}
{{\bf Proof:}}
{\qquad \hspace*{\fill} $\Box$}

\newcommand{\inner}{\operatorname{int}}

\newcommand{\FC}{\mathcal{F}}
\newcommand{\VC}{\mathcal{V}}
\newcommand{\MC}{\mathcal{M}}

\newcommand{\RC}{\mathcal{R}}

\newcommand{\N}{\mathbb{N}}

\newcommand{\R}{\mathbb{R}}

\begin{document}

\title{Weak convergence of robust functions on topological groups}
	\author{V\'{\i}ctor Ayala \thanks{%
			Supported by Proyecto Fondecyt $n^{o}$ 1190142, Conicyt, Chile} \ and
		Heriberto Rom\'{a}n-Flores \thanks{Corresponding Author's e-mail: heriberto.roman@gmail.com} \\
		Universidad de Tarapac\'a\\
		Instituto de Alta Investigaci\'on\\
		Casilla 7D, Arica, Chile\\
		and\\
		Adriano Da Silva \thanks{%
			Supported by Fapesp grant n%
			${{}^o}$
			2018/10696-6.}\\
		Instituto de Matem\'atica,\\
		Universidade Estadual de Campinas\\
		Cx. Postal 6065, 13.081-970 Campinas-SP, Brasil.\\
	}
\date{\today }
\maketitle

\begin{abstract}
	In this paper we introduce a weak version of level and epigraph convergence for level functions on topological spaces. In the particular case of topological groups we are able to define convolutions in the set of level functions and show that any such function is the limit in level and epigraph of robust functions.
\end{abstract}

	\textbf{Key words:} Level functions, robust functions, topological groups
	
	\textbf{2010 Mathematics Subject Classification: 54A20, 26E25, 54H11}

\section{Introduction}

The study of level-convergence and epigraphic-convergence of functions and their applications has been done by many authors,
including Rom\'{a}n-Flores et al. \cite{roman,roman1,roman2,roman3,roman4} in the setting of convergence of fuzzy
sets on finite-dimensional spaces, level-convergence of functions on regular
topological spaces, and compactness of spaces of fuzzy sets on a metric
space, respectively; Fang et al. \cite{fang} in level-convergence of fuzzy numbers; Greco et al. \cite{greco,greco1} in variational convergence of fuzzy
sets, and characterization of relatively compact sets of fuzzy sets on metric spaces; and Attouch \cite{attouch} in calculus of variations.

The main tools involved in these studies are based on Kura\-towski li\-mits and their connections with important variational properties. We recall that one of the most relevant pro\-per\-ties of the epi-convergence is the preservation of maximum (minimum)
points in epi-convergent sequences of functions. This explains the success of these convergence schemes in the global optimization
theory (see \cite{attouch}). In the setting of global optimization, in \cite{zheng,zheng1, zheng2} Zheng introduced the
concept of robust function as a generalization of upper semicontinuous functions. For robust functions the problem of global minimization on compact sets have an integral approach allowing the creation of an algorithm for the problem. 

The aim of the present paper is two-fold. First, to introduce a weak version of level and epi-convergence on topological spaces and to study it. The main difference between these convergences is the existence of a generalized type of minimum which basically gives us information about the behaviour of the function around but not at the point. Second, to study robust functions defined on topological groups. The main advantage of this case is the great generality it provides. By defining convolutions on the set of level functions, we are able to prove that any such function is in fact the limit of robust level functions, which could be of great interest in optimization problems. 

The paper is organized as follows: In Section 2 we provide the basic tools that will be used in the article concerning limit of subsets, level functions, epigraphs, etc. In Section 3, we introduce the concept of weak level and epigraph convergence and prove several properties concerning them. The main result of this section shows that both concepts are equivalent if the limit function is level continuous. In Section 4 we analyze the case where our topological space is a Hausdorff topological group. In this context we are able to show that any level function is the limit of robust functions.

\section{Preliminaries}

 We use this section to introduce the basic concepts needed in the rest of the paper. We also prove some results concerning the main properties of level functions.

\subsection{Convergence of sequence of subsets}

This section is concerned with the convergence of nets of subspaces of a given topological spaces. For more on the subject the reader could consult \cite[Chapter 3]{klein}.

Let $(X, \tau)$ be a topological space and $(x_{\lambda})_{\lambda\in\Lambda}$ a net in $X$. Let $\VC_x:=\{U\in\tau; \;\; x\in U\}$ be the set of neighborhoods of $x$. 

For metric spaces $X,Y$ a function $f:X\rightarrow Y$ is continuous if and only if for any $x\in X$  
\[
x_{n}\rightarrow x\Rightarrow f(x_{n})\rightarrow f(x)\text{ for every
	sequence }(x_{n})_{n\in \mathbb{N}}.
\]
The same is not true for topological spaces. The notion of net introduced by
E. H. Moore and Herman L. Smith in \cite{Moore} generalize the notion of a sequence
and solve the problem.

In order to define a net we need first the following notion. A nonempty set $%
\Lambda $ with a reflexive and transitive binary relation $\leq $ is a
direct set if given any $\lambda ,\beta \in \Lambda $ there exist $\gamma
\in \Lambda $ with $\lambda \leq \gamma $ and $\beta \leq \gamma $. A subset $\Lambda_0\subset \Lambda$
is said to be {\it cofinal} if for any $\lambda\in\Lambda$ there exists $\mu\in\Lambda_0$ such that $\mu\geq\lambda$.

\begin{definition}
	Let $X$ be a topological space, and $\Lambda $ a direct set. Any function $%
	f:\Lambda \rightarrow X$ is a net. We usually identify $f$ with its image $(x_{\lambda})_{\lambda \in \Lambda }$, where $x_{\lambda}:=f(\lambda)$.
\end{definition}

A point $x\in X$ is a {\it limit point} of $(x_{\lambda})_{\lambda\in\Lambda}$ if for every $U\in\VC_x$ there exists $\mu\in\Lambda$ such that $x_{\lambda}\in U$ for all $\lambda\geq\mu$. Also, we say that $x$ is a {\it cluster point} of $(x_{\lambda})_{\lambda\in\Lambda}$ if for every $U\in\VC_x$ and every $\mu\in\Lambda$ there is $\lambda\in\Lambda$ such that $\lambda\geq\mu$ and $x_{\lambda}\in U$.

\begin{definition}
	Let $(A_{\lambda})_{\lambda\in\Lambda}$ be a net of subsets of $X$.
	\begin{itemize}
		\item[1.] A point $x\in X$ is a limit point of $(A_{\lambda})_{\lambda\in\Lambda}$ if for every $U\in\VC_x$, there exists $\mu\in\Lambda$ such that $A_{\lambda}\cap U\neq\emptyset$ for all $\lambda\geq\mu$;
		\item[2.] A point $x\in X$ is a cluster point of $(A_{\lambda})_{\lambda\in\Lambda}$ if for every $U\in\VC_x$ and every $\mu\in\Lambda$ there exists $\lambda\in \Lambda$ with $\lambda\geq\mu$ and $A_{\lambda}\cap U\neq\emptyset$;
		\item[3.] $\liminf_{\lambda}A_{\lambda}$ is the set of all limit points of $(A_{\lambda})_{\lambda\in\Lambda}$;
		\item[4.] $\limsup_{\lambda}A_{\lambda}$ is the set of all cluster points of $(A_{\lambda})_{\lambda\in\Lambda}$;
		\item[5.] If $\limsup_{\lambda}A_{\lambda}=\liminf_{\lambda}A_{\lambda}=A$ we say that the net $(A_{\lambda})_{\lambda\in\Lambda}$ converges to $A$ and write $A=\lim_{\lambda}A_{\lambda}.$
	\end{itemize}
\end{definition} 

By \cite[Propositions 3.2.11 and 3.2.12]{klein} it holds that 
$$\limsup_{\lambda}A_{\lambda}=\bigcap_{\mu\in\Lambda}\overline{\bigcup_{\lambda\geq\mu}A_{\lambda}}\;\;\;\mbox{ and }\;\;\;\liminf_{\lambda}A_{\lambda}=\bigcap_{\Lambda_0}\overline{\bigcup_{\lambda\in\Lambda_0}A_{\lambda}},$$
where $\Lambda_0$ is a cofinal set in $\Lambda$. In particular $\liminf_{\lambda}A_{\lambda}$ and $\limsup_{\lambda}A_{\lambda}$ are closed subsets of $X$ and it holds that $\liminf_{\lambda}A_{\lambda}\subset \limsup_{\lambda}A_{\lambda}$. 

\begin{definition}
	We say that a net $(A_{\lambda})_{\lambda\in\Lambda}$ is {\it monotone increasing (resp. decreasing)} if 
	$$\lambda\leq\mu, \;\;\;\mbox{ implies }\;\;\; A_{\lambda}\subset A_{\mu}\;\;\Bigl(\mbox{resp. }\;\;A_{\lambda}\supset A_{\mu}\Bigr).$$
\end{definition}
 
The next result assures that for monotone nets the limit exists.

\begin{proposition}
	Let $(A_{\lambda})_{\lambda\in\Lambda}$ be a net of subsets of $X$. 
	\begin{itemize}
		\item[(i)] If $(A_{\lambda})_{\lambda\in\Lambda}$ is monotone increasing then $\lim_{\lambda}A_{\lambda}=\overline{\bigcup_{\lambda\in\Lambda}A_{\lambda}}$;
		\item[(ii)] If $(A_{\lambda})_{\lambda\in\Lambda}$ is monotone decreasing then $\lim_{\lambda}A_{\lambda}=\bigcap_{\lambda\in\Lambda}\overline{A_{\lambda}}$;
	\end{itemize}
\end{proposition} 

\begin{proof}
	Since the proof of both cases are similar, let us only show the monotone decreasing case. In this situation, it holds that 
	$$\forall\mu\in\Lambda, \;\;\;\bigcup_{\lambda\geq\mu}A_{\lambda}=A_{\mu}\;\;\;\implies\;\;\;\limsup_{\lambda}A_{\lambda}=\bigcap_{\mu\in\Lambda}\overline{A_{\mu}}.$$
	On the other hand, 
	$$\forall\lambda\in\Lambda, \;\;\;\bigcap_{\mu\in\Lambda}\overline{A_{\mu}}\subset \overline{A_{\lambda}}\;\;\;\implies\;\;\;\bigcap_{\mu\in\Lambda}\overline{A_{\mu}}\subset \overline{\bigcup_{\lambda\in\Lambda_0}A_{\lambda}},$$
	for any cofinal set $\Lambda_0$ of $\Lambda$. Hence 
	$$\limsup_{\lambda}A_{\lambda}=\bigcap_{\mu\in\Lambda}\overline{A_{\mu}}\subset\bigcap_{\Lambda_0}\overline{\bigcup_{\lambda\in\Lambda_0}A_{\lambda}}=\liminf A_{\lambda},$$
	which implies the result.
\end{proof}

Let $(\alpha_{\lambda})_{\lambda\in\Lambda}$ be a net of real numbers with $\alpha_{\lambda}\rightarrow\alpha$. In the sequel, we use the notation $\alpha_{\lambda}\nearrow$ (resp. $\alpha_{\lambda}\searrow$) when the net converges to $\alpha$ and is monotonic crecent (resp. decrescent) and there is no repetition of elements.

\subsection{Level functions}

Let $X$ be a topological space and consider $f:X\rightarrow [0, +\infty]$ a function. 

\begin{definition}
	For any given $\alpha>0$ the {\it $\alpha$-level sets} of $f$ reads as
	$$\ell_{\alpha}f:=\{x\in X; \;\;f(x)<\alpha\}\;\;\;\;\mbox{ and }\;\;\;\;L_{\alpha}f:=\{x\in X; \;\;f(x)\leq\alpha\}.$$
\end{definition}

We consider the set of {\it level functions} given by 
$$\FC(X):=\left\{f:X\rightarrow[0, +\infty], \;\;\ell_{\alpha}f\neq\emptyset, \;\forall\alpha>0\right\}.$$

The next proposition characterizes the level sets of a function $f\in \FC(X)$ by means of limits.

	\begin{proposition}
		\label{limits}
		For any $f\in\FC(X)$ and any $\alpha> 0$ it holds that
		$$\liminf_{\beta\rightarrow\alpha}L_{\beta}f=\overline{\ell_{\alpha}f}\subset \overline{L_{\alpha}f}\subset \bigcap_{\varepsilon>0}\overline{\ell_{\alpha+\varepsilon}f}= \limsup_{\beta\rightarrow\alpha}L_{\beta}f	;$$
		\end{proposition}
		
		\begin{proof}
			Since the inclusions 
			$$\overline{\ell_{\alpha}f}\subset \overline{L_{\alpha}f}\subset \bigcap_{\varepsilon>0}\overline{\ell_{\alpha+\varepsilon}f},$$
			follows directly from the definition of level sets, we will only show the equalities.
			
			Let then $x\in \ell_{\alpha}f$ and a net $\beta_{\lambda}\rightarrow \alpha$. There exists $\varepsilon>0$ such that $f(x)\leq \alpha-\varepsilon$. Hence,  $$\beta_{\lambda}\rightarrow\alpha\;\;\implies\;\;\exists \lambda_0\in\Lambda; \;\;\beta_{\lambda}\geq \alpha-\varepsilon,  \;\;\;\;\;\;\;\forall \lambda\geq \lambda_0,$$
			and so $f(x)\leq\alpha-\varepsilon\leq\beta_{\lambda}$ implying that $x\in L_{\beta_{\lambda}}f$ and showing that $x\in \liminf_{\beta\rightarrow\alpha}L_{\beta}f$. 
			Therefore, 
			$$\ell_{\alpha}f\subset \liminf_{\beta\rightarrow\alpha}L_{\beta}f\;\;\implies\;\;\overline{\ell_{\alpha}f}\subset \liminf_{\beta\rightarrow\alpha}L_{\beta}f.$$
			Reciprocally, let $x\in\liminf_{\beta\rightarrow\alpha}L_{\beta}f$ and consider a net $\beta_{\lambda} \nearrow\alpha$. By definition, 
			$$\exists x_{\lambda}\in L_{\beta_{\lambda}}f, \;\;\mbox{ such that }\;\;x_{\lambda}\rightarrow x.$$
			Therefore, $f(x_{\lambda})<\beta_{\lambda}<\alpha$ which implies $x_{\lambda}\in \ell_{\alpha}f$ implying that $x\in \overline{\ell_{\alpha}}f$ and so
			$\overline{\ell_{\alpha}f}= \liminf_{\beta\rightarrow\alpha}L_{\beta}f$ .
			
			Let us now consider $x\in\limsup_{\beta\rightarrow\alpha}L_{\beta}f$. By definition, there exists a net $\beta_{\lambda}\rightarrow \alpha$ and $x_{\lambda}\in L_{\beta_{\lambda}}f$ with $x_{\lambda}\rightarrow x.$ In particular, for any $\varepsilon>0$ there exists $\lambda_0\in\N$ such that $\beta_{\lambda}< \alpha+\varepsilon$ implying that 
			$$\forall \lambda\geq \lambda_0, \;\; f(x_{\lambda})\leq \beta_{\lambda}\leq \ell_{\alpha+\varepsilon}\;\;\implies \;\;x_{\lambda}\in \ell_{\alpha+\varepsilon}f\;\;\implies\;\;x\in\overline{\ell_{\alpha+\varepsilon}f},$$
			and hence $\limsup_{\beta\rightarrow\alpha}L_{\beta}f\subset\bigcap_{\varepsilon>0}\overline{\ell_{\alpha+\varepsilon}}f$. 
			
			Reciprocally, let $x\in \bigcap_{\varepsilon>0}\overline{\ell_{\alpha+\varepsilon}}f$ and consider $W\times I\subset U$ with $W\in\VC_{x}$ and $I\in\VC_{\alpha}$. The fact that $x\in \overline{\ell_{\alpha+\varepsilon}f}$ shows that $W\cap \ell_{\alpha+\varepsilon}f\neq\emptyset$ for all $\varepsilon>0$. In particular, there exists $\varepsilon_0>0$ such that $\alpha+\varepsilon_0\in I$. By taking $x_0\in W\cap \ell_{\alpha+\varepsilon_0}f$ and we get that $(x_0, \alpha+\varepsilon_0)\in W\times I$. Therefore, 
			$$\forall U\in\VC_{(x, \alpha)}, \;\;\exists (x_U, \beta_U)\in U\;\;\mbox{ such that }\;\; x_U\in \ell_{\beta_U}f,$$
			and hence, $(x_U, \beta_U)_{U\in\VC_{(x, \alpha)}}$ is a net such that $(x_U, \beta_U)\rightarrow (x, \alpha)$. It follows that $x\in\limsup_{\beta\rightarrow\alpha}L_{\beta}f$ and concluding the proof.	
			\end{proof}
			
\bigskip

\begin{example}
	Let $A\subset X$ and consider $\chi_A$ its characteristic function, that is,  
	$$\chi_A(x)=\left\{\begin{array}{c}
	0, \;x\in A\\ 1, \;x\notin A
	\end{array}\right..$$
	Then, 
	 $\ell_1\chi_A=A\;\;\;\mbox{ and }\;\;\;L_1f=X$ and hence
	 $$\overline{\ell_1\chi_A}=\overline{L_1\chi_A}\;\;\;\iff\;\;\; A\;\mbox{ is dense in }\;X.$$
\end{example}

\begin{definition}
	For any $f\in\FC(X)$ the {\it epigraphs} of $f$ reads as
	$$e(f):=\{(x, \alpha)\in X\times (0, +\infty); \;x\in \ell_{\alpha}f \}\;\;\mbox{ and }\;\;E(f):=\{(x, \alpha)\in X\times (0, +\infty); \;x\in L_{\alpha}f \}.$$
\end{definition}

The next result relates the topological properties of epigraphs and level sets

\begin{proposition}
	\label{Epi&level}
		With the previous notations, it holds:
	     
	     \begin{enumerate}
	     \item $\overline{E(f)}=\bigcup_{\alpha>0}\Bigl(\{\alpha\}\times\limsup_{\beta\rightarrow\alpha}L_{\beta}f\Bigr)$
		\item $\overline{\inner \ell_{\alpha}f}=\overline{\ell_{\alpha}f}, \;\;\;\;\forall\alpha>0\;\;\implies \;\;\;\; \overline{\inner e(f)}=\overline{e(f)}$;
	\end{enumerate}
\end{proposition}

\begin{proof}
	1. Let $(x, \alpha)\in \overline{E(f)}$. Then, for any $U\in\VC_x$ and $I\in\VC_{\alpha}$ it holds that $\left(U\times I\right)\cap E(f)\neq\emptyset.$
	In particular, by considering $I=(\alpha-\varepsilon, \alpha+\varepsilon)$ we get that $U\cap \ell_{\alpha+\varepsilon}\neq\emptyset$ and by Proposition \ref{limits} we obtain that 
	$$x\in\bigcap_{\varepsilon>0}\overline{\ell_{\alpha+\varepsilon}f}=\limsup_{\beta\rightarrow \alpha}L_{\beta}f\;\;\implies\;\;(x, \alpha)\in\left(\limsup_{\beta\rightarrow \alpha}L_{\beta}f\times\{\alpha\}\right)\;\;\implies\;\;\overline{E(f)}\subset\bigcup_{\alpha>0}\left(\limsup_{\beta\rightarrow \alpha}L_{\beta}f\times\{\alpha\}\right).$$
	On the other hand, let $x\in\limsup_{\beta\rightarrow\alpha}L_{\beta}f$ and consider $U\times I\in\VC_{(x, \alpha)}$. By Proposition \ref{limits} it holds that $U\cap\ell_{\alpha+\varepsilon}f\neq\emptyset$ for all $\varepsilon>0$. Therefore, by considering $\varepsilon>0$ such that $\alpha+\varepsilon\in I$ and $y\in U\cap\ell_{\alpha+\varepsilon}$ we get
	$$f(y)<\alpha+\varepsilon\;\;\implies\;\; (y, \alpha+\varepsilon)\in (U\times I)\cap E(f)\;\;\implies\;\; (x, \alpha)\in \overline{E(f)}\;\;\implies\;\; \bigcup_{\alpha>0}\left(\limsup_{\beta\rightarrow \alpha}L_{\beta}f\times\{\alpha\}\right)\subset \overline{E(f)}.$$
	
	2. Let $(x, \alpha)\in e(f)$ and consider $U\times I\in\VC_{(x, \alpha)}$. Then, 
	$$f(x)<\alpha\;\;\implies \;\;\exists\varepsilon>0; \;\; f(x)<\alpha-\varepsilon\;\;\mbox{ and }\;\; (\alpha-\varepsilon, \alpha+\varepsilon)\subset I.$$
	Also
	$$\overline{\inner\ell_{\alpha-\varepsilon}f}=\overline{\ell_{\alpha-\varepsilon}f}\;\;\implies\;\; U\cap \inner\ell_{\alpha-\varepsilon}f\neq\emptyset.$$
	On the other hand, the set 
	$$V:=U\cap \inner\ell_{\alpha-\varepsilon}f\times (\alpha-\varepsilon, \alpha+\varepsilon)\;\;\mbox{ is open and is contained in }U\times I.$$
	Moreover,
	$$(y, \beta)\in V\;\;\implies\;\; f(y)<\alpha-\varepsilon<\beta\;\;\implies \;\;V\subset \inner e(f)\;\;\implies\;\; (U\times I)\cap\inner e(f),$$
	implying that 
	$$e(f)\subset \overline{\inner e(f)}\;\;\mbox{ and hence }\;\;\overline{(f)}=\overline{\inner e(f)}.$$
	\end{proof}

\subsection{Generalized minimum and level continuity}

In this section we define the concept of minimum values for a level function. As we will see this notion will be important for convergence.

	\begin{definition}
		\label{min}
	 A point $x\in X$ is an {\it $\alpha$-generalized local minimum} of a given function $f\in \FC(X)$ if
		\begin{equation}
		\label{generalizedminimum}
		x\in \limsup_{\beta\rightarrow\alpha}L_{\beta}f\setminus\liminf_{\beta\rightarrow\alpha}L_{\beta}f.
		\end{equation}
		If in addition $f(x)=\alpha$ we say that $x$ is an {\it $\alpha$-local minimum} of $f$. We denote by $\MC_{\alpha}(f)$ the set of the $\alpha$-generalized minimum of $f$ and by $\MC(f)=\bigcup_{\alpha>0}\MC_{\alpha}(f)$ the set of generalized minimum of $f$.
	\end{definition}
	
	\begin{remark}
	 It is important to stress that a generalized minimum $x\in X$ gives us information about the behaviour of the graph of $f$ around it while a minimum gives us also information about the value $f(x)$. A function can have generalized minimum but not minimum as the next example shows.
	 \end{remark}
		
		\begin{example}
			
		Consider
		$$f:\R\rightarrow [0, +\infty], \;\;\;\mbox{ defined as }\;\;\;f(x):=\left\{\begin{array}{cc}
		0, \;&\; \mbox{ if }x\in [0, 1]\\
		-x+3, \;&\; \mbox{ if }x\in (1, 2)\\
		2, \;&\; \mbox{ if }x=2\\
		x-1, \;&\; \mbox{ if }x\in (2, +\infty)\\
		\end{array}\right.$$	
		
		A simple calculation shows that 
		$$\ell_1f=L_1f=[0, 1], \;\;\;\mbox{ and }\;\;\; \ell_{1+\varepsilon}f=[0, 1]\cup (2-\varepsilon, 2)\cup (2, 2+\varepsilon), \;\;\varepsilon\in (0, 1)$$
		implying that 
		$$\limsup_{\beta\rightarrow 1}L_{\beta}f=\bigcap_{\varepsilon>0}\overline{\ell_{1+\varepsilon}f}=[0, 1]\cup\{2\}\;\;\mbox{ and }\;\; \liminf_{\beta\rightarrow 1}L_{\beta}f=\overline{\ell_1f}=[0, 1].$$
		Since for $\alpha\neq 1, \;\limsup_{\beta\rightarrow\alpha}L_{\beta}f=\overline{\ell_{\alpha}f}$ then necessarily  
		$\MC(f)=\MC_1(f)=\{2\}$. Note also that $f$ has not minimum (see Figure \ref{fig1}).			
		\end{example}
		
		Next we define weak level continuity of a level function.
	
		\begin{definition}
			For any $\alpha>0$, a function $f\in\FC(X)$ is said to be {\it weak $\alpha$-level continuous} if $\MC_{\alpha}(f)=\emptyset$. 
			We say that $f$ is {\it weak level continuous} if it is weak $\alpha$-level continuous for all $\alpha>0$.	
		\end{definition}
		
		It is straightforward to see that $f\in \FC(X)$ is weak level continuous iff 
		$$\forall\alpha>0, \;\;\;\lim_{\beta\rightarrow\alpha}L_{\beta}f=\overline{L_{\alpha}f}.$$

\section{Convergence by level and by epigraph}

	In this section we introduce the concepts of level and epigraph convergence. We also analyze conditions for the equivalence of both concepts.	
	
	\begin{definition}
	 Let $(f_{\lambda})_{\lambda\in\Lambda}\subset \FC(X)$ be a net. We say that $f_{\lambda}$ {\it weak converges by level ($L$-converges)} to a function $f\in \FC(X)$ (or simply $f_{\lambda}\xrightarrow{\text{L}}f$) if 
	 $$\forall \alpha>0, \;\;\;\;\;L_{\alpha}f_{\lambda}\rightarrow \overline{L_{\alpha}f}.$$ 
	
	Analogously, a net $(f_{\lambda})_{\lambda\in\Lambda}\subset \FC(X)$ {\it weak converges by epigraph ($E$-converges)} to a function $f\in \FC(X)$ (or simply $f_{\lambda}\xrightarrow{\text{E}}f$) when 
	 $$E(f_{\lambda})\rightarrow \overline{E(f)}.$$ 
	 	\end{definition}
	  
	 We say that the function $f\in \FC(X)$ is a $L$-limit of the net $(f_{\lambda})_{\lambda\in\Lambda}$ if $f_{\lambda}\xrightarrow{L}f$ and and $E$-limit
	  if $f_{\lambda}\xrightarrow{E}f$.
	  
	 The next example shows that a net can be $E$-convergent but not $L$-convergent. 
	  
	  \begin{example}
	  	\label{ex}
	Let $(X, \Sigma, \mu)$ be a measure spaces and $L^p(X, \nu)$ its associated Banach space. Let $f_0\in L^p(X, \nu)$ with $\|f_0\|_p>1$ and consider $(\varepsilon_{\lambda})_{\lambda\in\Lambda}\subset (0, 1)$ a decreasing net such that $\varepsilon_{\lambda}\searrow0$. Define
	$$F_{\lambda}(f):=\left\{\begin{array}{cc}
	0, \;&\mbox{ if }\; f=0\\
	1-\varepsilon_{\lambda}\;&\mbox{ if }\; f\in B(f_0, \varepsilon_{\lambda})\\
	\|f-f_0\|_p+1\;&\mbox{ otherwise }
	\end{array}\right.\;\;\mbox{ and }\;\;F(f):=\left\{\begin{array}{cc}
	0, \;&\mbox{ if }\; f=0\\
	2, \;& \mbox{ if }\; f=f_0\\
	\|f-f_0\|_p+1\;&\mbox{ otherwise }
	\end{array}\right.,$$
	where $B(f_0, \varepsilon_{\lambda}):=\{f\in L^p(X, \nu); \;\|f-f_0\|_p<\varepsilon_{\lambda}\}$ is the $\varepsilon_{\lambda}$ ball in $L^p(X, \nu)$ centered at $f_0$.
	Since 
	$$B(f_0, \varepsilon_{\lambda})\supset B(f_0, \varepsilon_{\mu}), \;\;\;\mbox{ if }\;\;\;\lambda\leq\mu,$$
	we have that 
	$$\lim_{\lambda}L_{1}F_{\lambda}=\bigcap_{\lambda}\overline{\left(B(f_0, \varepsilon_{\lambda})\cup\{0\}\right)}=\{0, f_0\}.$$
	On the other hand, $\overline{L_1F}=L_1f=\{0\}$ showing that $F_{\lambda}$ does not $L$-converges to $F$. 
	$$E(f)=\{0\}\times(0, +\infty)\cup \{f_0\}\times[2, +\infty)\cup \bigcup_{\alpha\in(1, +\infty)}\left(\overline{B(f_0, \alpha)}\times\{\alpha\}\right)$$
	and so 
	$$\overline{E(f)}=\{0\}\times(0, +\infty)\cup \bigcup_{\alpha\in[1, +\infty)}\left(\overline{B(f_0, \alpha)}\times\{\alpha\}\right)$$
	
	Also,
	
	$$E(f_{\lambda})=\{0\}\times(0, +\infty)\cup \overline{B(f_0, \varepsilon_{\lambda})}\times [1-\varepsilon_{\lambda}, +\infty)\cup \bigcup_{\alpha\in [1+\varepsilon_{\lambda}, +\infty)}\left(\overline{B(f_0, \alpha)}\times\{\alpha\}\right),$$
	implying that 
	$$\lim_{\lambda}E(f_{\lambda})=\{0\}\times(0, +\infty)\cup \bigcup_{\alpha\in[1, +\infty)}\left(\overline{B(f_0, \alpha)}\times\{\alpha\}\right)=\overline{E(f)},$$
	and hence $f_{\lambda}\xrightarrow{E} f$ (see Figure \ref{fig2}).
	  \end{example}

	\begin{remark}
		A simple calculation shows us that 
			$$L_{\alpha}F_{\lambda}=\left\{\begin{array}{cc}
			\{0\}, \;&\mbox{ if }\;\alpha\in(0, 1-\varepsilon_{\lambda})\\
			\overline{B(f_0, \varepsilon_{\lambda})}\cup \{0\} \;&\mbox{ if }\; \alpha\in [1-\varepsilon_{\lambda}, 1]\\
			\overline{B(0, \alpha)}\cup\{0\}, \;&\mbox{ if }\; \alpha>1\}
			\end{array}\right.$$
			implying that the functions $F_{\lambda}$ in Example \ref{ex} are also lower semicontinuous.
	\end{remark}
	  
	  The next lemma relates the limits of epigraphs and level sets. It will be important in the proof of our main result.
	
	\begin{lemma}
		\label{closure}
		For all $\alpha>0$, it holds: 
		\begin{itemize}
			\item[1.] $x\in\limsup_{\lambda}L_{\alpha}f_{\lambda} \; \;\implies \;\;\; (x, \alpha)\in \limsup_{\lambda}E(f_{\lambda});$
			\item[2.]$x\in\liminf_{\lambda}L_{\alpha}f_{\lambda}\; \;\implies \;\;\; (x, \alpha)\in \liminf_{\lambda}E(f_{\lambda});$	 		
			\item[3.] $(x, \alpha)\in \limsup_{\lambda}E(f_{\lambda})\;\;\implies \;\;\forall\varepsilon>0, \;\;\;x\in\limsup_{\lambda}L_{\alpha+\varepsilon}f_{\lambda};$
			 \item[4.] $(x, \alpha)\in \liminf_{\lambda}E(f_{\lambda})\;\;\;\implies \;\;\; \forall\varepsilon>0,\;\;x\in\liminf_{\lambda}L_{\alpha+\varepsilon}f_{\lambda}.$
		\end{itemize}
	\end{lemma}
		
		\begin{proof}
			Since the items 1. and 3. are analogous to 2. and 4. respectively, we will only show 1. and 4.
			
			1. Let $x\in \limsup_{\lambda}L_{\alpha}f_{\lambda}$. By definition, there exists a subnet $\lambda_{\mu}\rightarrow+\infty$ and $x_{\lambda_{\mu}}\in L_{\alpha}f_{\lambda_{\mu}}$ with $x_{\lambda_{\mu}}\rightarrow x$. Consequently, 
			$$(x_{\lambda_{\mu}}, \alpha)\in E(f_{\lambda_{\mu}})\;\;\mbox{ and }\;\;(x_{\lambda_{\mu}}, \alpha)\rightarrow (x, \alpha)\;\;\implies\;\;(x, \alpha)\in\limsup_{\lambda}E(f_{\lambda}).$$

			4. Let $(x, \alpha)\in\liminf_{\lambda}E(f_{\lambda})$ and $\varepsilon>0$. For any $\lambda_{\mu}\rightarrow+\infty$,
			$$\exists \;(x_{\lambda_{\mu}}, \alpha_{\lambda_{\mu}})\in E(f_{\lambda_{\mu}});\;\;\mbox{ such that }\;\; (x_{\lambda_{\mu}}, \alpha_{\lambda_{\mu}})\rightarrow (x, \alpha).$$
			In particular, $\alpha_{\lambda_{\mu}}\rightarrow\alpha$ implies the existence of $\mu_0\in\Lambda$ such that $\alpha_{\lambda_{\mu}}<\alpha+\varepsilon$ if $\mu\geq \mu_0$ and hence
			$$(x_{\lambda_{\mu}}, \alpha_{\lambda_{\mu}})\in E(f_{\lambda_{\mu}})\;\;\implies\;\; x_{\lambda_{\mu}}\in L_{\alpha_{\lambda_{\mu}}}f\subset L_{\alpha+\varepsilon}f_{\lambda_{\mu}}, \;\mbox{ if }\; \mu\geq\mu_0,$$
			showing that $x\in\liminf_{\lambda}L_{\alpha+\varepsilon}f$ and finishing the proof.			
		\end{proof}
		
		\bigskip
		
		Now we can state and prove our main result concerning the level and the epigraph convergence.

	\begin{theorem}
		Let $(f_{\lambda})_{\lambda\in\Lambda}\subset \FC(X)$ be a net and $f\in \FC(X)$. Then,
		$$f_{\lambda}\xrightarrow{\text{L}} f\;\;\implies f_{\lambda}\xrightarrow{\text{E}} f.$$
		Reciprocally, if $f$ is level continuous, then
		$$f_{\lambda}\xrightarrow{\text{E}} f\;\;\implies f_{\lambda}\xrightarrow{\text{L}} f.$$
	\end{theorem}
	
	\begin{proof}
		We have to show that
		$$f_{\lambda}\xrightarrow{\text{L}} f\;\;\;\implies\;\;\; \limsup_{\lambda} E(f_{\lambda})\subset \overline{E(f)}\subset\liminf_{\lambda} E(f_{\lambda}).$$
		However, by Lemma \ref{closure} 
		$$(x, \alpha)\in \limsup_{\lambda} E(f_{\lambda})\;\;\implies\;\; x\in\limsup_{\lambda} L_{\alpha+\varepsilon}f_{\lambda},\;\;\forall\varepsilon>0,$$
		and hence 
		$$f_{\lambda}\xrightarrow{L} f\;\;\;\implies \;\;\;x\in\overline{L_{\alpha+\varepsilon}f}\subset \overline{\ell_{\alpha+2\varepsilon}f}, \;\;\forall\varepsilon>0.$$
		By Propositions \ref{limits} and \ref{Epi&level} we conclude that 
		$$x\in \bigcap_{\varepsilon>0}\overline{l_{\alpha+2\varepsilon}f}=\limsup_{\beta\rightarrow\alpha}L_{\beta}f\;\;\implies\;\;(x, \alpha)\in\left(\limsup_{\beta\rightarrow\alpha}L_{\beta}f\times\{\alpha\}\right)\subset \overline{E(f)},$$
		thus $\limsup_{\lambda} E(f_{\lambda})\subset \overline{E(f)}$.
		
		Let us consider now $(x, \alpha)\in E(f)$. Then, $x\in L_{\alpha}f$ and the assumption $f_{\lambda}\xrightarrow{L}f$ together with Lemma \ref{closure} imply 
		$$x\in L_{\alpha}f\subset\overline{L_{\alpha}f}=\liminf_{\lambda}L_{\alpha}f_{\lambda}\;\;\;\implies (x, \alpha)\in \liminf_{\lambda}E(f_{\lambda}).$$
		Therefore, $E(f)\subset \liminf_{\lambda}E(f_{\lambda})$, so $\overline{E(f)}\subset \liminf_{\lambda}E(f_{\lambda})$ and then $f_{\lambda}\xrightarrow{E}f$.
		
		\bigskip
		
		Let us now assume that $f_{\lambda}\xrightarrow{E}f$ with $f$ is level continuous. By definition,
		$$f_{\lambda}\xrightarrow{L}f\;\;\;\iff\;\;\; \forall\alpha>0, \;\;\;\limsup_{\lambda}L_{\alpha}f_{\lambda}\subset \overline{L_{\alpha}f}\subset\liminf_{\lambda}L_{\alpha}f_{\lambda}.$$ 
		
		Let $\alpha>0$ and $x\in \limsup_{\lambda}L_{\alpha}f_{\lambda}$. Since we are assuming that $f_{\lambda}\xrightarrow{E}f$, we have by Lemma \ref{closure} and Proposition \ref{Epi&level} that, 
		$$(x, \alpha)\in \limsup_{\lambda}E(f_{\lambda})=\overline{E(f)}=\bigcup_{\alpha>0}\left(\limsup_{\beta\rightarrow\alpha}L_{\beta}f\times\{\alpha\}\right).$$
		Therefore, $x\in \limsup_{\beta\rightarrow\alpha}L_{\beta}f$ and since we are assuming that $f$ is level continuous,
		$$x\in\limsup_{\beta\rightarrow\alpha}L_{\beta}f=\overline{L_{\alpha}f}\;\;\;\mbox{ which implies }\;\;\;\limsup_{\lambda}L_{\alpha}f_{\lambda}\subset \overline{L_{\alpha}f}.$$
		
		Consider now $x\in \ell_{\alpha}f$ and define $\alpha_0:=f(x)<\alpha$, then $(x, \alpha_0)\in E(f)$. Since we are assuming $f_{\lambda}\xrightarrow{E}f$, we get from Lemma \ref{closure} that
		$$(x, \alpha_0)\in E(f)\subset\liminf_{\lambda}E(f_{\lambda})\;\;\implies\;\;\forall\varepsilon>0, \;\; x\in \liminf_{\lambda}L_{\alpha_0+\varepsilon}f.$$
		Hence, for $\varepsilon>0$ small enough, we obtain
		$$L_{\alpha_0+\varepsilon}f_{\lambda}\subset L_{\alpha}f_{\lambda}\;\;\;\implies\;\;\; x\in\liminf_{\lambda}L_{\alpha_0+\varepsilon}f\subset \liminf_{\lambda}L_{\alpha}f_{\lambda}\implies \ell_{\alpha}f\subset \liminf_{\lambda}L_{\alpha}f_{\lambda}.$$ 
		On the other hand, we are assuming that $f$ is level-continuous, in particular $\overline{L_{\alpha}f}=\overline{\ell_{\alpha}f}$ and consequently
		$$\overline{L_{\alpha}f}=\overline{\ell_{\alpha}f}\subset \liminf_{\lambda}L_{\alpha}f_{\lambda}$$
		which implies that $f_{\lambda}\xrightarrow{L}f$ ending the proof.
	\end{proof}
	
\bigskip

Next we define monotone increasing nets.
	
	\begin{definition}
		We say that a net $(f_{\lambda})_{\lambda\in\Lambda}\subset\FC(X)$ is {\it monotone increasing} if the net $(E(f_{\lambda}))_{\lambda\in\Lambda}$ is monotone increasing.
	\end{definition}
	
	The next lemma states the main properties of monotone increasing nets.
	
	\begin{lemma}
		\label{monotone}
		For any net $(f_{\lambda})_{\lambda\in\Lambda}\subset\FC(X)$ it holds:
		\begin{itemize}
			\item[1.] $(f_{\lambda})_{\lambda\in\Lambda}$ is monotone increasing iff $(L_{\alpha}f_{\lambda})_{\lambda\in\Lambda}$ is monotone increasing for all $\alpha>0$.
			\item[2.] If $(f_{\lambda})_{\lambda\in\Lambda}$ is monotone increasing then, for all $x\in X$ the net
			$(f_{\lambda}(x))_{\lambda\in\Lambda}$ is monotone decreasing.
		\end{itemize}
	\end{lemma}
	
	\begin{proof}
		1. In this case, 	 
		$$\forall\alpha>0, \;\;\; x\in L_{\alpha}f_{\lambda}\;\;\implies\;\; (x, \alpha)\in E(f_{\lambda})\subset E(f_{\mu})  \;\;\implies\;\; x\in L_{\alpha}f_{\mu}.$$
		Reciprocally, if $\forall\alpha>0,\;\;\; L_{\alpha}f_{\lambda}\subset L_{\alpha}f_{\mu}$ then
		$$(x, \alpha)\in E(f_{\lambda})\;\;\implies\;\; x\in L_{\alpha}f_{\lambda}\subset L_{\alpha}f_{\mu}\;\;\implies\;\; (x, \alpha)\in E(f_{\mu}).$$
		
		2. Let $x\in X$ and $\lambda\leq\mu$. Then,  
		$$(x, f_{\lambda}(x))\in E(f_{\lambda})\subset E(f_{\mu})\;\;\implies\;\; f_{\mu}(x)\leq f_{\lambda}(x)\;\;\implies\;\; (f_{\lambda}(x))_{\lambda\in\Lambda}\;\;\mbox{ is decreasing.}$$
	\end{proof} 
	
	The next theorem shows that monotone increasing nets are $E$-convergents.
	
	\begin{theorem}
		Any monotone increasing net has an $E$-limit in $\FC(X)$.
	\end{theorem}
	
	\begin{proof} By Lemma \ref{monotone} we have that $(f_{\lambda}(x))_{\lambda\in\Lambda}$ is a monotone decreasing sequence. By the Monotone Converge Theorem for real sequences, we obtain that $(f_{\lambda}(x))_{\lambda\in\Lambda}$ converges to $f(x):=\inf_{\lambda\in\Lambda}f_{\lambda}(x)$. Moreover, 
		$$f_{\lambda}\in\FC(X), \;\;\lambda\in\Lambda\;\;\implies\;\; f\in\FC(X).$$ 
		Also,  
		$$(f_{\lambda})_{\lambda\in\Lambda}\;\;\mbox{ increasing }\;\;\implies \;\;\lim_{\lambda}E(f_{\lambda})=\overline{\bigcup_{\lambda\in\Lambda} E(f_{\lambda})}\;\;\mbox{ and hence }\;\;f_{\lambda}\xrightarrow{E} f\;\;\iff\;\; \overline{E(f)}=\overline{\bigcup_{\lambda\in \Lambda}E(f_{\lambda})}.$$
		If $(x, \alpha)\in E(f_{\lambda})$ for some $\lambda\in\Lambda$ then $f(x)\leq f_{\lambda}(x)\leq\alpha$ and hence 
		$$\overline{\bigcup_{\lambda\in\Lambda}E(f_{\lambda})}\subset \overline{E(f)}.$$
		On the other hand, let $(x, \alpha)\in E(f)$. For any $U\in \VC_{(x, \alpha)}$ there exists $\alpha_U>\alpha$ such that $(x, \alpha_U)\in U$. Hence,
		$$f(x)\leq\alpha<\alpha_U\;\;\implies\;\;\exists \lambda\in\Lambda; \;\; f_{\lambda}(x)<\alpha_U \;\;\implies\;\;(x, \alpha_U)\in E(f_{\lambda})\subset \bigcup_{\lambda\in\Lambda}E(f_{\lambda}).$$
		Hence, the net $\bigl((x, \alpha_U)\bigr)_{U\in\VC_{(x, \alpha)}}$ is contained in $\bigcup_{\lambda\in\Lambda}E(f_{\lambda})$ and $(x, \alpha_U)\rightarrow (x, \alpha)$ implying that 
		$$E(f)\subset \overline{\bigcup_{\lambda\in\Lambda}E(f_{\lambda})},$$
		and concluding the proof.		
	\end{proof}
	
	\begin{remark}
		An analogous definition of monotone decreasing nets is possible. However, there is no way to assure that a monotone decreasing net has a limit in $\FC(X)$.
	\end{remark}

	\section{Robust functions and topological groups}

This section is devoted to the study of robust functions on topological groups. Robust functions appears in optimization problems and therefore their understanding is desired (see for instance \cite{zheng, zheng1, zheng2}). Our aim here is to prove that any level function on a topological group is in fact the limit of robust function in both, level and epigraph convergence.

\begin{definition} 
	A subset $A\subseteq X$ is said to be robust iff $\overline{A}=\overline{\inner A}$.
\end{definition}

We define the class of {\it $L$-robust functions} of $\FC(X)$ as
$$\RC_L(X):=\{f\in\FC(X); \;\ell_{\alpha}f\;\mbox{ is robust }\;\forall\alpha>0\}.$$
and 
the class of {\it $E$-robust functions} of $\FC(X)$ as
$$\RC_E(X):=\{f\in\FC(X); \;e(f)\;\mbox{ is robust}\}.$$

By Proposition \ref{Epi&level} it holds that $\RC_L(X)\subset \RC_E(X)$.

\subsection{Topological groups}

Let $G$ be a topological Hausdorff group. For any $x\in G$, the {\it right-translation} by $x$ is the map 
$$R_x:G\rightarrow G, \;\;\;y\in G\mapsto yx\in G.$$
Is a standard fact that $R_x$ is a homeomorphism of $G$ with inverse given by $R_{x^{-1}}$, where $x^{-1}$ is the unique element in $G$ such that $xx^{-1}=x^{-1}x=e$, with $e\in G$ the identity element.

For any given nonempty subsets $A,B\subseteq G$ we define the set 
	$$AB=\{ab\;/\;a\in A,\;b\in B\}.$$

			\begin{proposition}
				\label{proposition}
				Let $A$ and $B$ be two nonempty subsets of $G$. It holds:
				\begin{itemize}
					\item[1.] If $A$ is open, then $AB$ is open;
					\item[2.] If $\inner A\neq \emptyset$, then $\left(\inner A\right)B\subseteq \inner AB$;
					\item[3.] If $A$ is robust, then $AB$ is robust.
					\end{itemize}
					\end{proposition}
					
					\begin{proof}
						1. In fact, since
						$$AB=\bigskip \bigcup_{b\in B}Ab=\bigcup_{b\in B}R_b(A)$$
						and right translations are homeomorphisms we have that if $A$ is open then $AB$ is open.
						
						2. Follows directly from the previous equality.
						
						3. We only have to show that $\overline{AB}\subset \overline{\inner AB}$ since the opposite inclusion always holds. Let then $x\in \overline{AB}$ and consider a neighborhood $U$ of $x$. By definition, 
						$$U\cap AB\neq\emptyset\;\;\implies\;\; ab\in U, \;\mbox{ for some }\;a\in A, b\in B.$$
						In particular $a\in Ub^{-1}$ and hence $Ub^{-1}$ is a neighborhood of $a$, since translations are homeomorphisms. By the assumption that $A$ is robust, we have that  $$\exists a'\in\inner A\;\mbox{ with }a'\in Ub^{-1}\;\;\implies\;\;a'b\in U\;\;\implies\;\;\left(\inner A\right)B\cap U\neq\emptyset,$$
						and by item 2. we conclude that $\inner AB\cap U\neq\emptyset$ and hence $x\in \overline{\inner AB}$ concluding the proof.
						\end{proof}

\subsection{$L$-robust functions on topological groups}

In this section we show that on topological groups, any function in $\FC(G)$ is the $L$-limit of some net $(f_{\lambda})_{\lambda\in\Lambda}\in\RC_L(G)$.

	Let $f, g\in\FC(G)$. The {\it L-convolution} of $f$ and $g$ is the function $f*_Lg\in\FC(G)$ given by 
	$$f*_Lg(x):=\inf_{y\in G}\{\max\{f(xy^{-1}), g(y)\}\}$$

	\begin{lemma}
		\label{product}
		For all $f, g \in\FC(G)$ and $\alpha>0$ it holds that
		\begin{itemize}
			\item[1.] $\ell_{\alpha}(f*_Lg)=\ell_{\alpha}f \, \ell_{\alpha}g$;
			\item[2.] $L_{\alpha}f\, L_{\alpha}g\subset L_{\alpha}(f*_Lg);$
			\item[3.] If there exists $\varepsilon_0>0$ such that $\ell_{\alpha+\varepsilon}g$ is open for all $\varepsilon\in(0, \varepsilon_0)$ then 
			$$L_{\alpha}(f*_Lg)\setminus\MC_{\alpha}(f)\subset \bigcap_{\varepsilon\in(0, \varepsilon_0)}\left(\ell_{\alpha}f\,\ell_{\alpha+\varepsilon}g\right).$$
		\end{itemize} 
	\end{lemma}
	
	\begin{proof}
		1. Let $x\in \ell_{\alpha}(f*g)$. Then,
		$$f*g(x)<\alpha\;\;\implies \;\;\inf_{y\in G}\{\max\{f(xy^{-1}), g(y)\}\}<\alpha\;\;\implies\;\;\exists y\in G; \;\;\max\{f(xy^{-1}), g(y)\}<\alpha$$
		$$\implies\;\; \exists y\in G; \;\;f(xy^{-1})>\alpha\;\mbox{ and }\; g(y)<\alpha\;\;\implies \;\;xy^{-1}\in \ell_{\alpha}f\;\mbox{ and }\;y\in \ell_{\alpha}g\;\;\implies\;\; x=(xy^{-1})y\in \ell_{\alpha}f \ell_{\alpha},g$$
		which proves that $\ell_{\alpha}(f*_Lg)\subset\ell_{\alpha}f \,\ell_{\alpha}g$.
		
		Reciprocally, if $a\in \ell_{\alpha}f$ and $b\in \ell_{\alpha}g$  then 
		$$\max\left\{f\left((ab)b^{-1}\right), g(b)\right\}<\alpha\;\;\implies \;\;f*g(ab)=\inf_{y\in G}\left\{\max\left\{f\left((ab)y^{-1}\right), g(y)\right\}\right\}<\alpha,$$
		showing that $\ell_{\alpha}f \,\ell_{\alpha}g\subset \ell_{\alpha}(f*g)$ and so $\ell_{\alpha}(f*_Lg)=\ell_{\alpha}f \, \ell_{\alpha}g$.
		
		2. It follows analogously from the inclusion $\ell_{\alpha}f \,\ell_{\alpha}g\subset \ell_{\alpha}(f*_Lg)$.
		For any 
		
		3. By definition, for any $\varepsilon>0$ there exists $y\in G$ such that 
		$$f*_Lg(x)+\varepsilon>\max\{f(xy^{-1}), g(y)\}\;\;\implies\;\;f(xy^{-1})<\alpha+\varepsilon\;\;\mbox{ and }\;\;g(y)<\alpha+\varepsilon$$
		$$\implies\;\; xy^{-1}\in \ell_{\alpha+\varepsilon}f\;\;\mbox{ and }\;\; y\in\ell_{\alpha+\varepsilon}g\;\;\iff\;\; x\left(\ell_{\alpha+\varepsilon}g\right)^{-1}\cap\ell_{\alpha+\varepsilon}f\neq\emptyset.$$
		In particular, we get that 
		$$x\left(\ell_{\alpha+\varepsilon}g\right)^{-1}\cap\ell_{\alpha+\delta}f\neq\emptyset,  \;\;\;\forall \varepsilon\in(0, \varepsilon_0)\;\mbox{ and }\;\delta>0.$$
		However, by hypothesis $x\left(\ell_{\alpha+\varepsilon}g\right)^{-1}\in\VC_x$ for any $\varepsilon\in (0, \varepsilon_0)$ and hence
		$$x\in L_{\alpha}(f*_Lg)\setminus\MC_{\alpha}(f)\;\;\implies\;\;x\left(\ell_{\alpha+\varepsilon}g\right)^{-1}\cap \ell_{\alpha}f\neq\emptyset, \;\;\forall\varepsilon\in (0, \varepsilon_0).$$
		Therefore, 
		$$x\in \ell_{\alpha}\ell_{\alpha+\varepsilon}g, \;\;\forall\varepsilon\in(0, \varepsilon_0)\;\;\implies\;\; x\in \bigcap_{\varepsilon\in (0, \varepsilon_0)}\left(\ell_{\alpha}f\,\ell_{\alpha+\varepsilon}g\right),$$
		which ends the proof.		
	\end{proof}
	
	\bigskip
	
	Next we prove that all functions in $\FC(G)$ are $L$-limits of robust functions.
	
	\begin{theorem}
		For any $f\in\FC(G)$ there exists a net $(f_{\lambda})_{\lambda\in\Lambda}\subset\RC_L(G)$ such that 
		$$f_{\lambda}\xrightarrow{L}f.$$
	\end{theorem}
	
	\begin{proof}
		For any $U\in\VC_e$ let us consider the indicator function of $U$ given by
		$$g_U(x):=\left\{\begin{array}{ll}
		0, & \mbox{ if }x\in U\\
		+\infty & \mbox{ if }x\notin U
		\end{array} \right.\;\mbox{ and define }\;f_U(x):=\left\{\begin{array}{ll}
		f*_Lg_U(x), & \mbox{ if }x\notin\MC(f)\\
		f(x),  &  \mbox{ if }x\in\MC(f)
		\end{array}\right..$$
		
		For all $\alpha>0$ it holds that,
		$$x\in\ell_{\alpha}f_U\;\;\iff\;\; f*_Lg(x)<\alpha \;\;\mbox{ or }\;\; f(x)<\alpha\;\;\iff x\in \ell_{\alpha}(f*_Lg_U)\;\;\mbox{ or }\;\; x\in\ell_{\alpha}f.$$
		However, by definition $\ell_{\alpha}g_U=U$ which by Lemma \ref{product} implies that $\ell_{\alpha}(f*_Lg)=\ell_{\alpha}f \,U$ and since $U\in\VC_e$ we obtain that $\ell_{\alpha}f\subset \ell_{\alpha}f\,U.$ By Proposition \ref{proposition} it follows that 
		$$\ell_{\alpha}f_U=\ell_{\alpha}(f*g_U)=\ell_{\alpha}f \,U\;\;\implies\;\;f_U\in\RC_L(G),$$		
		thus we only have to show that $(f_U)_{U\in\VC_e}$ $L$-converges to $f$, that is, 
		$$\limsup_UL_{\alpha}f_U\subset L_{\alpha}f\subset \liminf_{U}L_{\alpha}f_U, \;\;\;\forall\alpha>0.$$
		Let us first note that if $x\in\MC(f)$, by definition $f*g_U(x)=f(x)$ for all $U\in\VC_e$ and hence we only have to show the previous relation for $x\notin\MC(f)$.
		
		Consider $x\in \limsup_UL_{\alpha}f_U\setminus\MC_{\alpha}(f)$. We have, 
		$$\exists\{U_{\lambda}\}\subset\VC_e\;\;\mbox{ and }\;\;x_{\lambda}\in L_{\alpha}f_{U_{\lambda}}\;\;\mbox{ such that }\;\; \bigcap_{\lambda}U_{\lambda}=\{e\}\;\;\mbox{ and }\;\; x_{\lambda}\rightarrow x.$$ 
		If there exists a subnet $x_{\lambda_{\mu}}\in \MC_G(f)$ such that $x_{\lambda_{\mu}}\rightarrow x$, then 
		$$f(x_{\lambda_{\mu}})=f_{U_{\lambda}}(x_{\lambda_{\mu}})\leq\alpha\;\;\implies \;\;x_{\lambda_{\mu}}\in L_{\alpha}f\;\;\implies\;\;x\in\overline{L_{\alpha}f}.$$
		Therefore, we can assume w.l.o.g. that $x_{\lambda}\in L_{\alpha}f_{U_{\lambda}}\setminus\MC_{\alpha}(f)$ for all $\lambda$.
		In this case, the fact that $\ell_{\alpha+\varepsilon}g_{U_{\lambda}}=U_{\lambda}$ for all $\varepsilon>0$ implies by item 3. in Proposition \ref{product} that
		$$L_{\alpha}(f*_Lg_{U_{\lambda}})\setminus\MC_G(f)\subset \bigcap_{\varepsilon>0}\left(\ell_{\alpha}f\,\ell_{\alpha+\varepsilon}g_{U_{\lambda}}\right)=\ell_{\alpha}f\,U_{\lambda}.$$
		Therefore,
		$$\forall\lambda, \exists a_{\lambda}\in \ell_{\alpha}f\;\;\mbox{ and }\;\; b_{\lambda}\in U_{\lambda}; \;\;\mbox{ such that }\;\;x_{\lambda}=a_{\lambda}b_{\lambda}.$$
		However, $b_{\lambda}\in U_{\lambda}$ implies that $b_{\lambda}\rightarrow e$ and hence $a_{\lambda}=xb_{\lambda}^{-1}\rightarrow x$. Since $a_{\lambda}\in \ell_{\alpha}f$ we obtain that $x\in \overline{\ell_{\alpha}f}\subset \overline{L_{\alpha}f}$ implying that 
		$$\limsup_UL_{\alpha}f_U\subset\overline{L_{\alpha}f}.$$
		Consider now $x\in L_{\alpha}f\setminus\MC_{\alpha}(f)$ and a family $\{U_{\lambda}\}\subset\VC_e$ such that $\bigcap_{\lambda}U_{\lambda}=\{e\}$. By choosing $b_{\lambda}\in U_{\lambda}$ we have that $xb_{\lambda}\rightarrow x$ and, by item 2. in Proposition \ref{product}. that
		$$xb_{\lambda}\in L_{\alpha}fU_{\lambda}=L_{\alpha}fL_{\alpha}g_{U_{\lambda}}\subset L(f*_Lg_{U_{\lambda}})$$
		implying that $x\in \liminf_UL_{\alpha}(f*g_{U})$ and hence $f_U\xrightarrow{L} f$.			
	\end{proof}
						
\subsection{$E$-robust functions on topological groups}						

In this section we show that on topological groups it is also true that any function in $\FC(G)$ is the $E$-limit of some net $(f_{\lambda})_{\lambda\in\Lambda}\in\RC_E(G)$.						

Let us consider $G\times\R$ as a topological group, with the product given by 
$$(x_1, \alpha_1)(x_2, \alpha_2):=(x_1\cdot x_2, \alpha_1+\alpha_2).$$

Let $f, g\in\FC(G)$. The {\it E-convolution} of $f$ and $g$ is the function $f*_Eg\in\FC(G)$ given by 
$$f*_Eg(x):=\inf_{y\in G}\{f(xy^{-1})+g(y)\}.$$

\begin{lemma}
	\label{product2}
	For all $f, g \in\FC(G)$ and $\alpha>0$ it holds that
	\begin{itemize}
		\item[1.] $e(f*_Eg)=e(f)e(g)$;
		\item[2.] $ E(f)E(g)\subset E(f*_Eg)\subset \overline{E(f) E(g)}.$
		\item[3.] $e(f)$ and $E(f)$ are invariant for translations by elements in $\{e\}\times \R$.
	\end{itemize}
	
\end{lemma}

\begin{proof} 1. Let $(x, \alpha)\in e(f*g)$. In particular, if $\varepsilon>0$ is such that $f*_Eg(x)<\alpha-\varepsilon$ there exists $y\in G$ and such that 
	$$f(xy^{-1})+g(y)+\varepsilon<\alpha.$$
	Then, $(x_1, \alpha_1)=(xy^{-1}, f(xy^{-1})+\varepsilon)$ and $(x_2, \alpha_2)=(y, \alpha-f(xy^{-1})-\varepsilon)$ are such that
	$$(x_1, \alpha_1)\in e(g)\;\;\mbox{ and }\;\; (x_2, \alpha_2)\in e(g)$$
	and 
	$$(x_1, \alpha_1)(x_2, \alpha_2)=(x_1x_2, \alpha_1+\alpha_2)=(x, \alpha)\;\;\implies \;\;e(f*_Eg)\subset e(f)e(g).$$
	
	Reciprocally, if $(x_1, \alpha_1)\in e(f)$ and $(x_2, \alpha_2)\in e(g)$. Then $f(x_1)< \alpha_1$ and $g(x_2)<\alpha_2$ gives us
	$$\alpha_1+\alpha_2> f(x_1)+g(x_2)=f\left((x_1x_2)x^{-1}_2\right)+g(x_2)\geq \inf_{y\in G}\{f\left((x_1x_2)y^{-1}\right)+g(y)\}=(f*_Eg)(x_1x_2)$$
	implying that $(x_1, \alpha_1)(x_2, \alpha_2)=(x_1x_2, \alpha_1+\alpha_2)\in e(f*g)$ and hence 
	$$e(f)e(g)\subset e(f*_Eg).$$

	2. The inclusion $E(f)E(g)\subset E(f*_Eg)$ is analogous to the inclusion $e(f)e(g)\subset e(f*_Eg)$. 	Consider then $(x, \alpha)\in E(f*g)$. By definition, for any $\varepsilon>0$ there exists $y\in G$ such that 
	$$f(xy^{-1})+g(y)<\alpha+\varepsilon.$$
	By defining $(x_1, \alpha_1)=(xy^{-1}, f(xy^{-1}))$ and $(x_2, \alpha_2)=(y, \alpha+\varepsilon-f(xy^{-1}))$ we get that 
	$$(x_1, \alpha_1)\in E(f)\;\;\mbox{ and }\;\;(x_2, \alpha_2)\in E(g),$$
	furthermore
	$$(x_1, \alpha_1)(x_2, \alpha_2)=(x_1x_2, \alpha_1+\alpha_2)=(x, \alpha+\varepsilon)\;\;\implies\;\; (x, \alpha+\varepsilon)\in E(f)E(g),$$
	implying that 
	$$(x, \alpha)\in \overline{E(f)E(g)}\;\;\;\mbox{ and hence }\;\; E(f*g)\subset\overline{E(f)E(g)},$$
	which finishes the proof.
\end{proof}

The next result shows that any function in $f\in\FC(G)$ can is the $E$-limit of a net in $\RC_E(G)$.

\begin{theorem}
	For any $f\in\FC(G)$ there exists a net $(f_{\lambda})_{\lambda\in\Lambda}\subset\RC_E(G)$ such that 
	$$f_{\lambda}\xrightarrow{E}f.$$
\end{theorem}

\begin{proof}
	For any $U\in\VC_e$ let us consider indicator function of $U$ given by
	$$g_U(x):=\left\{\begin{array}{ll}
	0, & \mbox{ if }x\in U\\
	+\infty & \mbox{ if }x\notin U
	\end{array} \right.$$
Then, $e(g_U)=U\times(0, +\infty)$ and hence $g_U\in\RC(G)$. Moreover, by Lemma \ref{product2} it holds that 
$$e(f*g_U)=e(f) e(g_U)=e(f)(U\times\R^+),$$
and so, Proposition \ref{proposition} assures that $e(f*g_U)$ is open and in particular robust.
We claim that $f_U\xrightarrow{E}f$, or equivalently,
$$\limsup_{U} E(f_U)\subset \overline{E(f)}\subset\liminf_UE(f_U).$$	
Let us consider $(x, \alpha)\in \limsup_{U} E(f_U)$. By definition, 
$$\exists \{U_{\lambda}\}\subset\VC_e\;\;\mbox{ and }\;\;(x_{\lambda}, \alpha_{\lambda})\in E(f_{U_{\lambda}}) \;\;\mbox{ such that }\;\; \bigcap_{\lambda}U_{\lambda}=\{e\}\;\;\mbox{ and }\;\;(x_{\lambda}, \alpha_{\lambda})\rightarrow (x, \alpha).$$	
Moreover, from Lemma \ref{product2} it holds that $E(f_U)\subset \overline{E(f)E(g_U)}$, we can assume w.l.o.g. that 
$$(x_{\lambda}, \alpha_{\lambda})\in E(f)E(g_{U_{\lambda}}),$$
so we can write 
$$(x_{\lambda}, \alpha_{\lambda})=(a_{\lambda}, \gamma_{\lambda})(b_{\lambda}, \beta_{\lambda}), \;\;\mbox{ with }\;\; (a_{\lambda}, \gamma_{\lambda})\in E(f), \;\mbox{ and }\; (b_{\lambda}, \beta_{\lambda})\in E(g_{U_{\lambda}}).$$
However,
$$(b_{\lambda}, \beta_{\lambda})\in e(g_{U_{\lambda}})=U_{\lambda}\times\R^+\;\;\implies\;\; b_{\lambda}\rightarrow e\;\;\implies\;\; a_{\lambda}=x_{\lambda}b_{\lambda}^{-1}.$$
On the other hand, the fact that $\alpha_{\lambda}=\gamma_{\lambda}+\beta_{\lambda}$ implies in particular that both $(\gamma_{\lambda})_{\lambda}$ and $(\beta_{\lambda})_{\lambda}$ are bounded nets in $\R$. By taking subnets if necessary we are able to assume w.l.o.g. that $\gamma_{\lambda}\rightarrow \gamma\;\;\mbox{ and }\;\;\beta_{\lambda}\rightarrow\beta$ implying that 
$$(a_{\lambda}, \gamma_{\lambda})\rightarrow (x, \gamma)\;\;\mbox{ and hence }\;\;(x, \gamma)\in \overline{E(f)}.$$
Again, from the fact that right translations are homeomorphisms in $G\times \R$ implies that 
$$(x, \alpha)=(a, \gamma+\beta)=(a, \gamma)(e, \beta)\in \overline{E(f)}(e, \beta)=\overline{E(f)(e, \beta)}\subset \overline{E(f)}$$
and hence $\limsup_{U} E(f_U)\subset \overline{E(f)}$.

Let us now consider $(x, \alpha)\in E(f)$ and a family of neighborhood $\{U_{\lambda}\}$ such that $\bigcap_{\lambda}U_{\lambda}=\{e\}$. By considering $b_{\lambda}\in U_{\lambda}$ and choosing $\alpha_{\lambda}\in(0, +\infty)$ such that $\alpha_{\lambda}\rightarrow 0$ we have that 
$$(x, \alpha)(b_{\lambda}, \alpha_{\lambda})\in E(f)(U_{\lambda}\times(0, +\infty))=E(f)E(g_{U_{\lambda}})\subset E(f*_Eg_{\lambda}).$$
On the other hand, 
$$(x, \alpha)(b_{\lambda}, \alpha_{\lambda})=(xb{\lambda}, \alpha+\alpha_{\lambda})\rightarrow (x, \alpha)$$
showing that $(x, \alpha)\in\liminf_{U}E(f*_Eg_U)$ and hence that $E(f)\subset \liminf_{U}E(f*_Eg_U)$ which implies necessarily that $f_U\xrightarrow{E} f$. 
\end{proof}

\pagebreak

\subsection{Figures}

\begin{figure}[h!]
	\begin{center}
		\includegraphics[scale=0.2, angle=270]{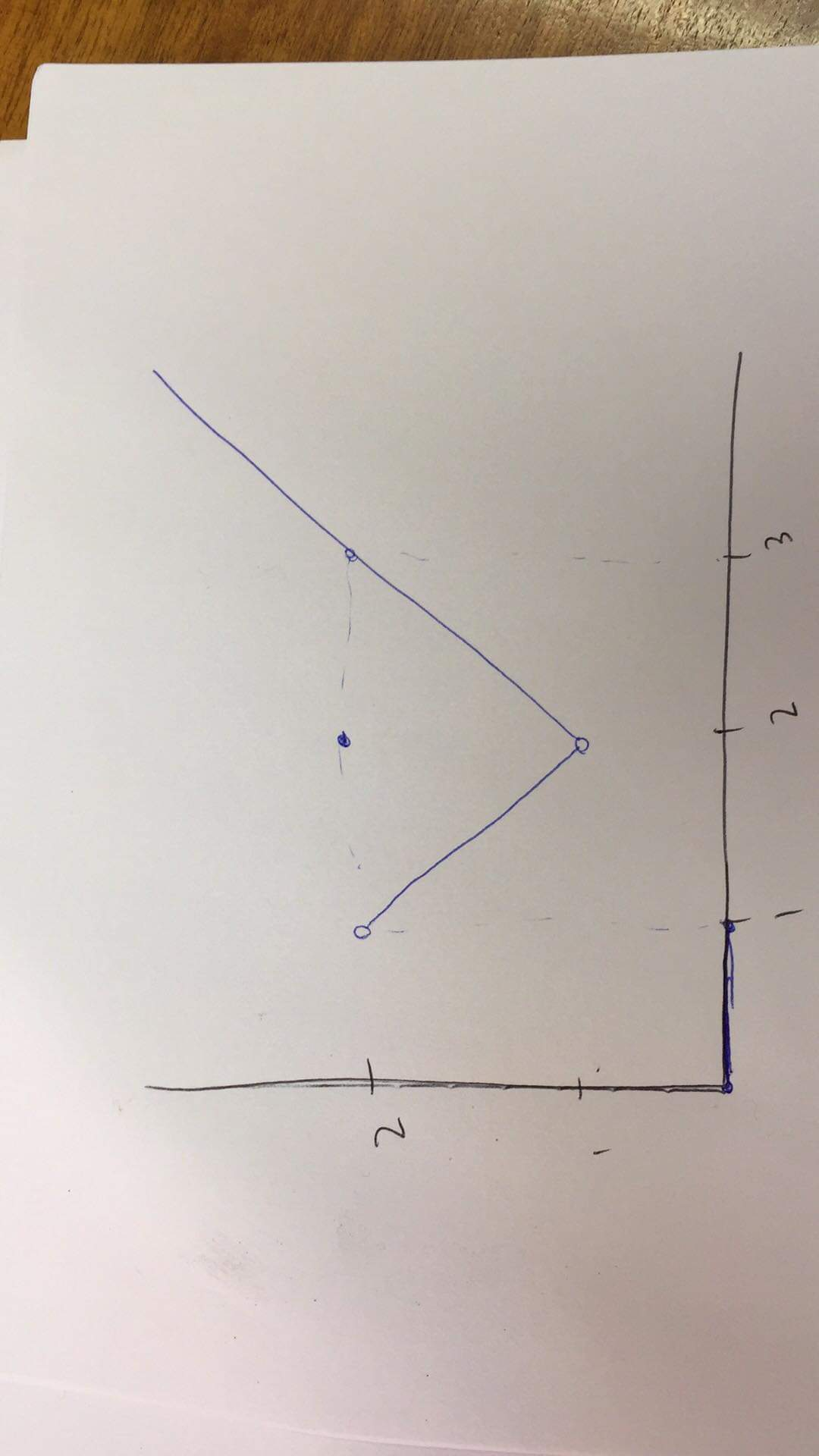}
	\end{center}
	\caption{Function with generalized minimum and without minimum.}
	\label{fig1}
\end{figure}

\begin{figure}[h!]
	\begin{center}
		\includegraphics[scale=0.2, angle=90]{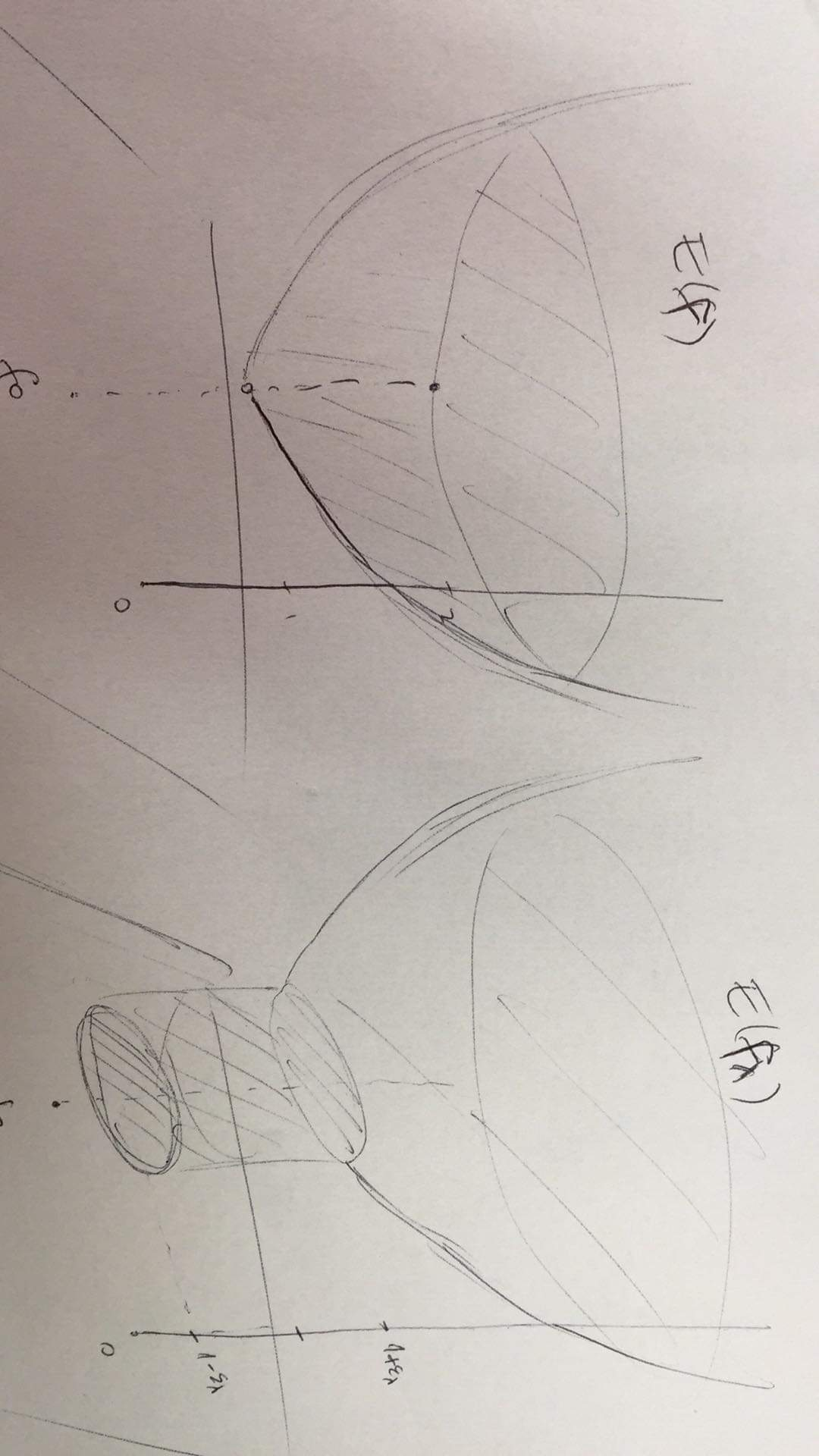}
	\end{center}
	\caption{Net that $E$-converges but not $L$-converges.}
	\label{fig2}
\end{figure}


\begin{thebibliography}{99}
	\bibitem{zheng}  
	\newblock Zheng, Q.,
	\newblock \emph{Robust analysis and global minimization of a class of discontinuous functions (I).}
	\newblock Acta Mathematicae Applicatae Sinica (English series), {\bf 6} 3, 205-223 (1990).
	
	\bibitem{zheng1}  
	\newblock Zheng, Q.,
	\newblock \emph{Robust analysis and global minimization of a class of discontinuous functions (II).}
	\newblock Acta Mathematicae Applicatae Sinica (English series), {\bf 6} 4, 317-337 (1990).

	
	\bibitem{zheng2}
	\newblock Zheng, Q.,
	\newblock \emph{Discontinuity and measurability of robust functions in the integral global optimization.}
	\newblock Computer \& Math. Applic. {\bf 25}, 79-88 (1993).
	
	\bibitem{roman}  
	\newblock Rojas-Medar, M., Rom\'{a}n-Flores, H.,
	\newblock \emph{Equivalence of convergences of fuzzy sets.}
	\newblock Fuzzy Sets and Systems {\bf 80}, 217-224 (1996).
	
	\bibitem{roman1} 
	\newblock Rom\'{a}n-Flores, H., Rojas-Medar, M.,
	\newblock\emph{Level continuity of functions and applications.}
	\newblock Computers \& Math. Applic. {\bf 38} (3/4), 143-149 (1999).
	
	\bibitem{roman2}  
	\newblock Rojas-Medar, M., Bassanezi, R., Rom\'{a}n-Flores, H.,
	\newblock\emph{A generalization of the Minkowski embedding theorem and applications.}
	\newblock Fuzzy Sets and Systems {\bf 102} , 263-269 (1999).
	
	\bibitem{roman3}  
	\newblock Rom\'{a}n-Flores, H., Rojas-Medar, M., 
	\newblock\emph{Embeddding of Level-continuous fuzzy sets on Banach spaces.}
	\newblock Information Sciences {\bf 144}, 227-242 (2002).
	
	\bibitem{roman4} 
	\newblock Rom\'{a}n-Flores, H., 
	\newblock\emph{The compactness of E(X).}
	\newblock Appl. Math. Lett. {bf 11}, 13-17 (1998).
	
	\bibitem{greco} 
	\newblock Greco, G., Moshen, M., Quelho, E., 
	\newblock\emph{On the variational convergence of fuzzy sets.}
	\newblock Ann. Univ. Ferrara (S. VII-Sc. Math.) {\bf 44}, 27-39 (1998).
	
	\bibitem{fang} 
	\newblock Fang, J., Huang, H., 
	\newblock\emph{On the level-convergence of a sequence of fuzzy numbers.}
	\newblock Fuzzy Sets and Systems {\bf 147}, 417-435 (2004).
	
	\bibitem{greco1} 
	\newblock Greco, G., 
	\newblock\emph{Sendograph metric and relatively compact sets of fuzzy sets.}
	\newblock Fuzzy Sets and Systems {\bf 157}, 286-291 (2006).
	
	\bibitem{attouch} 
	\newblock Attouch, H., 
	\newblock\emph{Variational convergence for Functions and Operators.}
	\newblock Pitman, London (1984).
	
	
	
	
	
	
	
	
	
	
	\bibitem{klein}  
	\newblock Klein, E., Thompson, A.C.,
	\newblock \emph{Theory of Correspondences.}
	\newblock Wiley, New York (1984)
	
	\bibitem{Moore}
	\newblock Moore, E. H.;\ Smith, H. L.,
	\newblock \emph{General Theory of Limits.}
	\newblock American Journal of Mathematics 44 (2); 102-121  (1922).
	
	
	
	
	
	
\end{thebibliography}
\end{document}